\documentclass{amsart}
\usepackage[english]{babel}
\usepackage{amsthm,hyphenat}
\usepackage{amsmath}
\usepackage{amssymb}
\usepackage[latin1]{inputenc}
\usepackage{mathrsfs}
\usepackage{dsfont}
\usepackage[bookmarksnumbered,colorlinks]{hyperref}
\usepackage{graphics}
\usepackage{enumerate}
\usepackage[all]{xy}
\newcommand{\R}{\mathds R}
\newcommand{\N}{\mathds N}

\newcommand{\Q}{{\mathcal Q}}
\newcommand{\Ps}{{\mathcal P}}

\title[Curvature computations in Finsler Geometry]{Curvature computations in Finsler Geometry using a distinguished class of anisotropic connections}

\author[M. A. Javaloyes]{Miguel \'Angel Javaloyes}
\address{Departamento de Matem\'aticas, \hfill\break\indent
Universidad de Murcia, \hfill\break\indent
Campus de Espinardo,\hfill\break\indent
30100 Espinardo, Murcia, Spain}
\email{majava@um.es}

\thanks{2000 {\em Mathematics Subject Classification:} Primary  53C50, 53C60\\
\textbf{Key words:} Anisotropic linear connections, Finsler Geometry, Jacobi operator, Bianchi Identities.}

\makeindex

\begin{document}
\newtheorem{thm}{Theorem}[section]
\newtheorem{prop}[thm]{Proposition}
\newtheorem{lemma}[thm]{Lemma}
\newtheorem{cor}[thm]{Corollary}
\theoremstyle{definition}
\newtheorem{defi}[thm]{Definition}
\newtheorem{notation}[thm]{Notation}
\newtheorem{exe}[thm]{Example}
\newtheorem{conj}[thm]{Conjecture}
\newtheorem{prob}[thm]{Problem}
\newtheorem{rem}[thm]{Remark}

\begin{abstract}
We show how to compute tensor derivatives and curvature tensors using affine connections. This allows for all computations to be obtained without using coordinate systems, in a way that parallels the computations appearing in classical  Riemannian Geometry. In particular, we obtain Bianchi identities for the curvature tensor of any anisotropic connection, we compare the  curvature tensors of any two anisotropic connections, and we find a family of anisotropic connections which are well suited to study the geometry of Finsler metrics.
\end{abstract}

\maketitle

%\tableofcontents

\section{Introduction}
Traditionally, Finsler Geometry is associated with lengthy computations in coordinates. This is due to the dependence on directions of all the elements, which allows for a large generality of the metrics, but sometimes makes it difficult to understand the geometric meaning of certain quantities. 
In order to overcome these difficulties, we will use affine connections $\nabla^V$, which are defined for every vector field $V$ which is non-zero everywhere. The connections $\nabla^V$ can be interpreted as osculating affine connections in the same way as one obtains the osculating metric $g_V$ of a Finsler metric by fixing at every point $p\in M$ the direction of $V_p$ in the fundamental tensor, namely, $(g_V)_p=g_{V_p}$, where $g$ is the fundamental tensor in \eqref{fundten}.  This approach was first considered in \cite{Mat80,Rund}, later in \cite[\S 7]{Sh01} and recently in \cite{Jav14a,Ja14b,JaSo14}.

Here we will go a step further. First, we consider anisotropic connections in a manifold $M$, which are not exactly connections on fiber bundles, but especially adapted to the dependence on the direction (see Definition \ref{aniconnection} and \cite[\S 4.4]{Ja19} for the relationship with connections on the vertical bundle). Then we will use the anisotropic tensor calculus developed in \cite{Ja19}, and the formulas \eqref{productrule}, \eqref{partialT} and \eqref{RvRV}, wherein the derivative of a tensor and the curvature are computed using $\nabla^V$.
In order to take advantage of this approach, we make a fundamental observation in Proposition \ref{assump}: that there is a privileged choice of the extension $V$ which allows one to compute the derivative of a tensor. This choice has the property that at a fixed point $p$, the vector field $V$ is parallel in all directions, namely, $(\nabla^V_XV)_p=0$ for all vector fields $X$. This simplifies dramatically the computations involving curvature tensors and derivatives. Indeed, it reduces, for example, the proof of the Bianchi identities to the classical case of an affine connection in a manifold, \S \ref{bianchiid}. It also allows us to relate the curvature tensors of two different anisotropic connections using the difference tensor, \S \ref{curvdiff}. In particular, this relation will lead us to distinguish a family of connections which are well suited for studying Finsler metrics, \S \ref{goodcon}. Amongst these connections, one finds the Berwald and the Chern connections, and for all of them, it is possible to derive formulas for the first and the second variations of the energy (Prop. \ref{variations}), determining the same Jacobi operator, Jacobi equation and flag curvature (Prop. \ref{Jacobioper}). Moreover, these connections can also be related with the Levi-Civita connection of the osculating metric (Prop. \ref{oscul}).

The paper is organized as follows. In \S \ref{anicon}, we give the basic notions of anisotropic tensor calculus, previously introduced in \cite{Ja19}. In particular, we define anisotropic tensors, anisotropic connections,  and finally the tensor derivation  and the curvature tensor  associated with an anisotropic connection. In  \S \ref{covariant}, we give the notion of anisotropic covariant derivation and then of auto-parallel curve. We also establish the Jacobi equation of an auto-parallel curve in Prop. \ref{Jacoboper} and give a condition in terms of the difference tensor (see \eqref{differencetensor}), which implies that two different anisotropic connections determine the same Jacobi operator. In \S \ref{parallel}, we explain the different possibilities for parallel transport with an anisotropic connection. In subsection \S \ref{bianchiid}, we obtain the anisotropic  Bianchi identities and in \S \ref{curvdiff},  the comparison of the curvature tensors of two anisotropic connections. Section \ref{goodcon} is devoted to the study of certain connections which are well suited to study Finsler metrics. These connections allow us to  obtain formulas for the first and the second variation of the energy. 
\section{Anisotropic tensor calculus and affine connections}\label{anicon}

Let $M$ be a smooth manifold of dimension $n$, $TM$ its tangent bundle and $T^*M$ its contangent bundle, with $\pi:TM\rightarrow M$ and $\tilde\pi :T^*M\rightarrow M$, the natural projections. Given  an open subset $A$ of the tangent bundle $TM$ with $\pi(A)=M$, we can use the restriction $\pi|_A:A\subset TM\rightarrow M$ to obtain two {\em pull-back vector bundles } over $A$ by lifting $\pi$ and $\tilde \pi$, which are respectively denoted by $\pi_A^*:\pi_A^*TM\rightarrow A$ and $\tilde \pi_A^*:\pi^*_AT^*M\rightarrow A$:
\begin{displaymath}
\xymatrix{ \pi_A^*TM \ar[d]_{\pi^*_A} & TM \ar[d]^{\pi} \\
	A\subset TM \ar[r]^{\quad\,\,\pi_A=\pi|_A}  & M  }
\quad \quad \xymatrix{ \pi_A^*T^*M \ar[d]_{\tilde\pi^*_A} & T^*M \ar[d]^{\tilde\pi} \\
	A\subset TM \ar[r]^{\quad\,\,\pi_A=\pi|_A}  & M  }
\end{displaymath}
 Observe that for every $v\in A$, we have that $(\pi_A^*)^{-1}(v)=T_{\pi(v)}M$ and $(\tilde \pi_A^*)^{-1}(v)=T^*_{\pi(v)}M$. Then a section of $\pi^*_A$ (resp. $\tilde\pi^*_A$) can be thought as a smooth map $A\ni v\rightarrow X_v\in TM$ (resp. $A\ni v\rightarrow \theta_v\in T^*M$) in such a way that $X_v\in T_{\pi(v)}M$ (resp. $\theta_v\in T^*_{\pi(v)}M$). The subset of (smooth) sections of $\pi^*_ATM$ will be denoted by ${\mathfrak T}^1_0(M,A)$, while the subset of smooth sections of $\tilde\pi^*_AT^*M$ will be denoted by ${\mathfrak T}^0_1(M,A)$. 
Then we define an {\em $A$-anisotropic tensor $T$ of type $(r,s)$, $r,s\in \N\cup \{0\}$, $r+s>0$, }  as an ${\mathcal F}(A)$-multilinear map 
\[T:{\mathfrak T}^0_1(M,A)^r\times{\mathfrak T}^1_0(M,A)^s\rightarrow {\mathcal F}(A), \]
where ${\mathcal F}(A)$ is the subset of smooth real functions on $A$, namely, $f:A\rightarrow \R$. The space of $A$-anisotropic tensors of type $(r,s)$ is denoted by ${\mathfrak T}^r_s(M,A)$, while by convention ${\mathfrak T}^0_0(M,A)\equiv {\mathcal F}(A)$. The ${\mathcal F}(A)$-multilinearity implies that for every $v\in A$, $T$ determines a multilinear map 
\[T_v:(T^*_{\pi(v)}M)^r\times (T_{\pi(v)}M)^s\rightarrow \R. \]
%where $T_v(\theta_1,)
As a consequence, given an open subset $\Omega\subset M$,   it makes sense to consider the restriction \[T:{\mathfrak T}^0_1(\Omega,T\Omega\cap A)^r\times{\mathfrak T}^1_0(\Omega,T\Omega\cap A)^s\rightarrow {\mathcal F}(T\Omega\cap A).\] In particular, given a system of coordinates $(\Omega,\varphi)$, where $\Omega$ is an open subset of $M$ and $ \varphi:\Omega\rightarrow U\subset \R^n$, a chart of $M$, we define the coordinates of $T$ as functions $T_{j_1j_2\ldots j_s}^{i_1i_2\ldots i_r}:A\cap T\Omega\rightarrow \R$ defined as
\[T_{j_1j_2\ldots j_s}^{i_1i_2\ldots i_r}(v)=T_v(dx^{i_1},\ldots,dx^{i_r},\partial_{j_1},\ldots,\partial_{j_s}),\]
where $\partial_1,\ldots,\partial_n$ denotes the frame of partial vector fields associated with the coordinate system $(\Omega,\varphi)$ and $dx^1,\ldots,dx^n$, its dual basis.
 Observe that  the space of smooth vector fields on $M$, denoted by ${\mathfrak X}(M)$ (resp. the space of smooth one-forms on $M$, denoted by ${\mathfrak X}^*(M)$) can be viewed as a subset of ${\mathfrak T}^1_0(M,A)$ (resp. ${\mathfrak T}^0_1(M,A)$),  since a vector field $X\in {\mathfrak X}(M)$ (resp. $\theta\in{\mathfrak X}^*(M)$) can be identified with the smooth section $\tilde {\mathcal X}$ (resp. $\tilde\theta$) defined as $ \tilde {\mathcal X}_v=X_{\pi(v)}$ (resp. $\tilde\theta_v=\theta_{\pi(v)}$).
By the ${\mathcal F}(A)$-multilinearity, it is enough to define the tensor as
\begin{equation}\label{Treduction}
T:{\mathfrak X}^*(M)^r\times{\mathfrak X}(M)^s\rightarrow {\mathcal F}(A), 
\end{equation}
which then will be extended by the  ${\mathcal F}(A)$-multilinearity using a local frame in ${\mathfrak X}(M)$ (resp. ${\mathfrak X}^*(M)$), see also \cite[Remark 2]{Ja19}. 

One can also consider an ${\mathcal F}(A)$-multilinear map 
\begin{equation}\label{forthecurv}
T: \mathfrak{T}^1_0(M,A)^s\rightarrow
\mathfrak{T}^1_0(M,A),
\end{equation} 
which determines the $A$-anisotropic tensor of type $(1,s)$
$\bar{T}:\mathfrak{T}^0_1(M,A)\times
\mathfrak{T}^1_0(M,A)^s\rightarrow {\mathcal F}(A)$ 
defined by  
\begin{equation}\label{interpr}
\bar{T}(\theta,{\mathcal X}_1,\ldots,{\mathcal X}_s)=\theta(T({\mathcal X}_1,\ldots,{\mathcal X}_s)).
\end{equation}
 As in classical tensor calculus, $T$ will be considered as  a tensor field itself, using the formula above only when necessary.

%\begin{defi}\label{partialT}
	We will say that a vector field $V$ defined on an open subset $\Omega\subset M$ is {\em $A$-admissible} if $V_p\in A$ for every $p\in \Omega$.
%\end{defi}
In such a case, given an $A$-anisotropic tensor $T\in\mathfrak{T}^r_s(M,A)$, we can define a (classical) tensor $T_V\in \mathfrak{T}^r_s(\Omega)$ in such a way that $(T_V)_p=T_{V_p}$ for every $p\in \Omega$. %as a map
%\[T_V:{\mathfrak X}^*(\Omega)^r\times {\mathfrak X}(\Omega)^s\rightarrow {\mathcal F}(\Omega),\]
%such that \[T_V(\theta^1,\ldots,\theta^r,X_1,\ldots,X_s)(p)=
%T_{V(p)}(\theta^1,\ldots,\theta^r,X_1,\ldots,X_s),\] where ${\mathcal F}(\Omega)=\{f:\Omega\rightarrow \R: f\in C^\infty\}$, and ${\mathfrak X}(\Omega)$ and ${\mathfrak X}^*(\Omega)$ denote, respectively, the space of vector fields and one-forms on $\Omega$.
%\subsection{Vertical derivations} 

As a result of the dependence on directions of $A$-anisotropic tensors, one can define  derivatives on the vertical bundle.
\begin{defi}\label{defi:vert}
	Given an $A$-anisotropic tensor $T\in \mathfrak{T}^{r}_s(M,A)$, we define its {\em vertical derivative} as the tensor  $\partial^\nu T\in \mathfrak{T}^{r}_{s+1}(M,A)$ given by
	\[(\partial^\nu T)_v(\theta^1,\ldots,\theta^r,X_1,\ldots,X_s,Z)=\frac{\partial}{\partial t} T_{v+tZ_{\pi(v)}}(\theta^1,\ldots,\theta^r,X_1,\ldots,X_s)|_{t=0}\]
	for any $v\in A$ and  $(\theta^1,\theta^2,\ldots,\theta^r,X_1,\ldots,X_s,Z)\in \mathfrak{X}^*(M)^{r}\times \mathfrak{X}(M)^{s+1}$, and an analogous definition is made for $A$-anisotropic tensors of the type \eqref{forthecurv}.  % (recall Remark \ref{extendT}).
\end{defi}
Recall that in Finsler Geometry, the linear connections used to study geodesics and curvature are linear connections on the vertical bundle. Along this paper we will use a different notion of connection introduced in \cite[\S 7.1]{Sh01} and studied in \cite{Ja19}, which simplifies some computations.

\begin{defi}\label{aniconnection}
	An  {\em$A$-anisotropic (linear) connection} is  a map
	\[\nabla: A\times \mathfrak{X}(M)\times\mathfrak{X}(M)\rightarrow TM,\quad\quad (v,X,Y)\mapsto\nabla^v_XY%:=
	%\nabla(v, X,Y)
	\in T_{\pi(v)}M,\]
	such that
	\begin{enumerate}[(i)]
		\item $\nabla^v_X(Y+Z)=\nabla^v_XY+\nabla^v_XZ$, for any $X,Y,Z\in  {\mathfrak X}(M)$,
		\item $\nabla^v_X(fY)=X(f) Y_{\pi(v)}+f(\pi(v)) \nabla^v_XY $ for any $f\in {\mathcal F}(M)$, $X,Y\in  {\mathfrak X}(M)$,
		\item for any $X,Y\in  {\mathfrak X}(M)$,  $\nabla_XY\in {\mathfrak T}^1_0(M,A)$ (considered as a map $A\ni v\rightarrow \nabla^v_XY$),
		\item $\nabla^v_{fX+hY}Z=f(\pi(v))\nabla^v_XZ+h(\pi(v)) \nabla^v_YZ$, for any $f,h\in {\mathcal F}(M)$, $X,Y,Z\in  {\mathfrak X}(M)$.
	\end{enumerate}
\end{defi}
 For the relation of this new notion of $A$-anisotropic  connection with classical linear connections see \cite[\S 4.4]{Ja19}.  Given an $A$-anisotropic connection $\nabla$ and a vector field $X\in {\mathfrak X}(M)$, it is possible to define an $A$-anisotropic tensor derivation $\nabla_X$  (see \cite[\S 2.2]{Ja19} for the general definition) in the space of tensors ${\mathfrak T}^r_s(M,A)$ such that for any function $h\in{\mathcal F}(A)$, $\nabla_Xh\in{\mathcal F}(A)$ is determined at $v\in A$ by
 \begin{equation}\label{derh}
\nabla_X h (v) = X_{\pi(v)}(h(V))-(\partial^\nu h)_v(\nabla_X^vV),
 \end{equation}
 where $V$ is any $A$-admissible vector field extending $v$, namely, $V_{\pi(v)}=v$. Observe that the expression in \eqref{derh} does not depend on the choice of $V$ (see \cite[Lemma 9]{Ja19}). Moreover, if $\theta\in {\mathfrak X}^*(M)$, then $\nabla_X\theta\in \mathfrak{T}^0_1(M,A)$  is determined by
 \begin{equation}\label{dertheta}
 (\nabla_X\theta)_v ({ Y})=X_{\pi(v)}(\theta( {Y}))-\theta(\nabla^v_X{Y}), \quad \text{for any $Y\in \mathfrak{X}(M)$.}
 \end{equation}
 Finally, for an arbitrary $A$-anisotropic tensor $T\in {\mathfrak T}^r_s(M,A)$, we define the tensor derivative
 \begin{align}\label{productrule}
(\nabla_XT)(\theta^1,\ldots,\theta^r,X_1,\ldots,X_s))=& 
 \nabla_X(T(\theta^1,\ldots,\theta^r,X_1,\ldots, X_s))\nonumber\\
 &-\sum_{i=1}^r T(\theta^1,\ldots,\nabla_X\theta^i,\ldots, \theta^r,
 X_1,\ldots,X_s)\nonumber\\
 &-\sum_{j=1}^s T(\theta^1,\ldots, \theta^r,
 X_1,\ldots,\nabla_XX_j,\ldots,X_s),
 \end{align}
 for any $(\theta^1,\theta^2,\ldots,\theta^r,X_1,\ldots,X_s)\in \mathfrak{X}^*(M)^{r}\times \mathfrak{X}(M)^{s}$ (see \cite[Theorem 11]{Ja19} and recall that $\nabla_X$ is an $A$-anisotropic derivation as in \cite[Definition 8]{Ja19}). Observe that the same formula \eqref{productrule} with $r=0$ also holds for tensors of the type \eqref{forthecurv}.
 We can also define the {\em torsion} of $\nabla$ as
 \begin{equation}\label{torsion}
 {\mathcal T}_v(X,Y)=\nabla^v_XY-\nabla^v_YX-[X,Y], \quad \text{for any $X,Y\in \mathfrak{X}(M)$.}
 \end{equation}
 We say that an $A$-anisotropic connection is {\em torsion-free} if $\mathcal T=0$.
\begin{rem}\label{extensions}
	Recall that even if $\nabla_X\theta$ in \eqref{dertheta}, $\nabla_XT$ in \eqref{productrule} and $\mathcal T$ in \eqref{torsion} are defined only for one-forms and vector fields, they can be extended to arbitrary elements of ${\mathfrak T}^0_1(M,A)$ and ${\mathfrak T}^1_0(M,A)$ by ${\mathcal F}(A)$-multilinearity. Moreover, $\nabla$ also can be extended to ${\mathfrak T}^1_0(M,A)\times {\mathfrak T}^1_0(M,A)$ using the Leibnitz rule and \eqref{derh}. One can also obtain the following formula, when given an $A$-admissible vector field $V$ on an open subset $\Omega$ and  ${\mathcal X,\mathcal Y}\in {\mathfrak T}^1_0(M,A)$,
	\begin{equation}\label{DY}
\nabla^v_{\mathcal X}{\mathcal Y}=	\nabla^v_{\mathcal X}({\mathcal Y}_V)-(\partial^\nu {\mathcal Y})_v(\nabla_{\mathcal X}^VV),
	\end{equation}
	where $(\partial^\nu {\mathcal Y})_v(z)=\left.\frac{d}{dt} {\mathcal Y}(v+tz)\right|_{t=0}$, for any vector $z\in T_{\pi(v)}M$ (see \cite[Eq. (12)]{Ja19}), and recall that $({\mathcal Y}_V)_p={\mathcal Y}_{V_p}$ for every $p\in \Omega$. When $T$ is an $A$-anisotropic tensor as in \eqref{forthecurv}, this can be used to compute the first term of $\nabla_XT$ in \eqref{productrule} with the help of the associated affine connection $\nabla^V$ for a given $A$-admissible vector field $V$ which extends $v\in A$. Namely, 
	\begin{equation}\label{partialT}
	\nabla_X(T(X_1,\ldots, X_s))(v)=X_{\pi(v)}(T_V(X_1,\ldots, X_s))-(\partial^\nu T)_v(X_1,\ldots,X_s,\nabla_X^VV),
	\end{equation}
where $X,X_1,\ldots,X_s\in {\mathfrak X}(M)$ (see \cite[Eq. (17)]{Ja19} for more details). 
 \end{rem}
Given a system of coordinates $(\Omega,\varphi)$, we will define the Christoffel symbols  of the $A$-anisotropic connection $\nabla$ as the functions $\Gamma^k_{\,\, ij}:T\Omega\cap A\rightarrow \R$ determined by
\[\nabla^v_{\partial_i}\partial_j=\Gamma^k_{\,\, ij}(v)\left(\partial_k\right)_{\pi(v)}.\]
It is easy to check that $\nabla$ is torsion-free, namely, $\mathcal T=0$, if and only if the Christoffel symbols $\Gamma^k_{\,\, ij}$ are symmetric in $i$ and $j$.
\subsection{The curvature tensor of an $A$-anisotropic connection} \label{curvanis}
Given an $A$-anisotropic  connection $\nabla$, we can define the associated curvature tensor $R_v:{\mathfrak X}(M)\times {\mathfrak X}(M)\times {\mathfrak X}(M)\rightarrow T_{\pi(v)}M$, as follows
\begin{equation}\label{Rv}
R_v(X,Y)Z=\nabla^v_X(\nabla_YZ)-\nabla^v_Y(\nabla_XZ)-
\nabla^v_{[X,Y]}Z,
\end{equation}
for any $v\in A$ and $X,Y,Z\in {\mathfrak X}(M)$ (recall part $(iii)$ of Def. \ref{aniconnection} and Remark~\ref{extensions} for the extension of $\nabla$ to  ${\mathfrak T}^1_0(M,A)\times {\mathfrak T}^1_0(M,A)$). It is straightforward to check that $R$ is an ${\mathcal F}(A)$-multilinear map, and then an $A$-anisotropic tensor as in \eqref{forthecurv}. Furthermore, it is anti-symmetric in $X$ and $Y$.
%In this expression, let us point out that we interpret $\nabla_YZ$ as an element of $\mathfrak{T}^1_0(M,A)$, in the sense that $\nabla_YZ(v)=\nabla^v_YZ$, and $\nabla^v_X$ as the extension of ${\mathcal D}=\nabla^v_X$ described in \eqref{DXdef}. It is not difficult to prove that $R_v$ is ${\mathcal F}(M)$-multilinear, and then it determines an anisotropic tensor. As in the case of the torsion, $R_v$ cannot evaluated directly in elements of $\mathfrak{T}^1_0(M,A)$, as the Lie bracket is not well-defined, but it can be extended by linearity.

Recall that given an
$A$-admissible vector field $V$ in $\Omega\subset M$, an $A$-anisotro\-pic connection $\nabla$ provides an affine connection $\nabla^V$ on $\Omega$ defined as $(\nabla^V_XY)_p=\nabla^v_XY$ for any $X,Y\in \mathfrak{X}(M)$, being $v=V_p$. Our next aim is to express the curvature tensor in terms of the elements associated with  $\nabla^V$. First, we need to introduce the following tensors:
\begin{align*}
P_v(X,Y,Z)&=\frac{\partial}{\partial t}\left(\nabla^{v+tZ(\pi(v))}_XY\right)|_{t=0},\\
R^V(X,Y)Z&=\nabla^V_X\nabla^V_YZ-\nabla^V_Y\nabla^V_XZ-
\nabla^V_{[X,Y]}Z,\end{align*}
where  $X,Y,Z\in{\mathfrak X}(\Omega)$. Observe that $P$ is an $A$-anisotropic tensor, but $R^V$ is not. This is because   $R^V$ does depend on the particular choice of $V$ as we will see later.  The $A$-anisotropic tensor $P$ will be called the {\it vertical derivative} of $\nabla$ and the connection $\nabla$ is said to be {\it Berwald} if and only if $P=0$.  Moreover, in a natural system of coordinates of the tangent bundle $(T\Omega,\tilde{\varphi})$, associated with  a coordinate system $(\Omega,\varphi)$ on $M$, one has
\begin{equation}\label{Pcoord}
P_v(u,w,z)= u^iw^j z^k \frac{\partial\Gamma_{\,\,ij}^{l}}{\partial y^k}(v)\left(\partial_l\right)_{\pi(v)}
\end{equation}
for every $v\in A$, and $u,w,z\in T_{\pi(v)}M$ and being $u^i, w^i$ and $z^i$ the coordinates of $u,w,z$. As usual, we denote the coordinates of a point $v\in T\Omega$ as 
\begin{equation}\label{naturalcoord}
\tilde\varphi=(x,y)=(x^1,x^2,\ldots,x^n,y^1,y^2,\ldots,y^n),
\end{equation}
and we use the {\it  Einstein summation convention} when possible, omitting the coordinate functions $\varphi$ and $\tilde\varphi$ to avoid clutter in equations. It follows from \eqref{Pcoord} that if $\nabla$ is torsion-free, then $P$ is symmetric in the first two components. % (recall Remark \ref{torsionfree}).
\begin{rem}\label{Pv=0}
If the $A$-anisotropic connection is positive homogeneous of degree zero, namely, $\nabla^{\lambda v}=\nabla^v$, then it follows that $P_v(u,w,v)=0$ for every $v\in A$ and $u,w\in T_{\pi(v)}M$.
\end{rem}
\begin{prop}\label{curv}
Let $\nabla$ be an $A$-anisotropic connection and $\Omega\subset M$, an open subset. Then for any $v\in A$, we have that
\begin{equation}\label{RvRV}
R_v(X,Y)Z=(R^V)_p(X,Y)Z-(P_V)_p(Y,Z,\nabla^V_XV)+(P_V)_p(X,Z,\nabla^V_YV),
\end{equation}
where $V,X,Y,Z\in {\mathfrak{X}}(\Omega)$, being $V$ an $A$-admissible extension of $v$ and $p=\pi(v)$. Moreover, in a natural system of coordinates $(T\Omega,\tilde\varphi)$ of $TM$, we have 
\begin{multline}
R_v(X,Y)Z=
\left[Z^i(p)Y^j(p)X^m(p) (\frac{\partial\Gamma_{\,\,ji}^k}{\partial x^m}
(v)-v^l\Gamma^h_{\,\, ml}(v)\frac{\partial \Gamma^k_{\,\, ji}}{\partial y^h}(v))\right.\\ \left.-Z^i(p)X^j(p) Y^m(p)(\frac{\partial \Gamma_{\,\,ji}^k}{\partial x^m}(v)-v^l\Gamma_{\,\,ml}^h(v)\frac{\partial \Gamma_{\,\,ji}^k}{\partial y^h}(v))\right.\\\left.
+Z^i(p)Y^j(p) X^m(p) \left(\Gamma_{\,\,ji}^l(v) \Gamma_{\,\,ml}^k(v)-\Gamma_{\,\,mi}^l(v) \Gamma_{\,\,jl}^k(v)\right)\right]\left(\partial_k\right)_{\pi(v)},\label{Firstcur2}
\end{multline}
where $X^i,Y^i,Z^i, v^i$ are the coordinates of $X,Y,Z,v$, respectively.
\end{prop}
\begin{proof}
In order to prove \eqref{RvRV}, it is enough to observe that using \eqref{DY}, we deduce that 
\begin{align*}
\nabla^v_X(\nabla_YZ)=(\nabla^V_X(\nabla^V_YZ))_{\pi(v)}-P_v(Y,Z,\nabla^V_XV),\\
\nabla^v_Y(\nabla_XZ)=(\nabla^V_Y(\nabla^V_XZ))_{\pi(v)}-P_v(X,Z,\nabla^V_YV).
\end{align*}
Let us now check \eqref{Firstcur2}.
Denote the Christoffel symbols of $\nabla^V$ as $\tilde{\Gamma}_{\,\,ij}^k(p)=\Gamma_{\,\,ij}^k(V_p)$. Then
\begin{multline}
 R^V(X,Y)Z=
 \left[Z^iY^jX^l \frac{\partial\tilde{\Gamma}_{\,\,ji}^k}{\partial x^l}
 -Z^iX^j Y^l\frac{\partial \tilde{\Gamma}_{\,\,ji}^k}{\partial x^l}\right.\\\left.
 +Z^iY^j X^m \left(\tilde{\Gamma}_{\,\,ji}^l \tilde{\Gamma}_{\,\,ml}^k-\tilde{\Gamma}_{\,\,mi}^l \tilde{\Gamma}_{\,\,jl}^k\right)\right]\frac{\partial}{\partial x^k}.\label{Firstcur}
 \end{multline}
 Moreover, as $(\nabla^V_XV)^k=X^m\frac{\partial V^k}{\partial x^m}+X^mV^l\Gamma_{\,\,ml}^k\circ V$, using \eqref{Pcoord}, we deduce that 
 \begin{align*}
 P_V(Y,Z,\nabla^V_XV)&=Z^iY^j(X^m\frac{\partial V^l}{\partial x^m}+X^mV^l\Gamma^h_{\,\, ml})\frac{\partial \Gamma^k_{\,\, ji}}{\partial y^h}\circ V,\\
 P_V(X,Z,\nabla^V_YV)&=Z^iX^j(Y^m\frac{\partial V^l}{\partial x^m}+Y^mV^l\Gamma^h_{\,\,ml})\frac{\partial \Gamma^k_{\,\, ji}}{\partial y^h}\circ V.
 \end{align*}
 Then, using the last equations, \eqref{Firstcur} and $\frac{\partial\tilde{\Gamma}_{\,\,ji}^k}{\partial x^m}(p)
 =\frac{\partial\Gamma_{\,\,ji}^k}{\partial x^m}(V_p)+\frac{\partial V^l}{\partial x^m}(p)\frac{\partial\Gamma_{\,\,ji}^k}{\partial y^l}(V_p)$, we finally obtain \eqref{Firstcur2}.
% \begin{multline}
% R_V(X,Y)Z=
% \left[Z^iY^jX^p (\frac{\partial\Gamma_{\,\,ji}^k}{\partial x^p}
% (V)-V^l\Gamma^h_{\,\, pl}(V)\frac{\partial \Gamma^k_{\,\, ji}}{\partial y^h}(V))\right.\\ \left.-Z^iX^j Y^p(\frac{\partial \Gamma_{\,\,ji}^k}{\partial x^p}(V)-V^l\Gamma_{\,\,pl}^h(V)\frac{\partial \Gamma_{\,\,ji}^k}{\partial y^h}(V))\right.\\\left.
% +Z^iY^j X^m \left(\Gamma_{\,\,ji}^l(V) \Gamma_{\,\,ml}^k(V)-\Gamma_{\,\,mi}^l(V) \Gamma_{\,\,jl}^k(V)\right)\right]\frac{\partial}{\partial x^k},\label{Firstcur2}
% \end{multline}
% from which follows the independence of $R$ from the choice of $V$ as the partial derivatives of $V$ vanish.
\end{proof}
%\begin{defi}
%Given an anisotropic (linear) connection $\nabla$, we say that the anisotropic tensor $R$ defined in Lemma \ref{curv} is the curvature tensor of $\nabla$. 
%\end{defi}
%We can obtain now the expression of the curvature tensor in a system of coordinates $(\Omega,\varphi)$. If
%$R_v(\partial_k,\partial_l)\partial_j=R_{j\,\, kl}^{\,\,\,i}(v)\partial_i$, then from \eqref{Firstcur2}, it follows that
%\[R_{j\,\, kl}^{\,\,\,i}=\frac{\delta \Gamma^i_{\,\, lj}}{\delta x^k}-\frac{\delta \Gamma^i_{\,\, kj}}{\delta x^l}+\Gamma^i_{\,\, kh}\Gamma^h_{\,\, lj}-
%\Gamma^i_{\,\, lh}\Gamma^h_{\,\, kj},
%\]
%where the operator $\delta$ is defined as %has been introduced at the end of $\S$
%$\frac{\delta \Gamma^i_{\,\, lj}}{\delta x^k}(v)=\frac{\partial \Gamma^i_{\,\, lj}}{\partial x^k}(v)-N^m_{\,\, k}(v)\frac{\partial \Gamma^i_{\,\, lj}}{\partial y^m}(v)$. %\ref{nonlinearity}.
\subsection{Covariant derivatives along curves}\label{covariant}
In the following, given a smooth curve $\gamma:[a,b]\rightarrow M$, $\mathfrak{X}(\gamma)$ will denote the space of smooth vector fields along $\gamma$ and ${\mathcal F}(I)$ the smooth real functions defined on $I=[a,b]$.
\begin{defi}\label{deficov}
An {\em $A$-anisotropic covariant derivation} $D^v_\gamma$ in $A$ along a curve $\gamma:[a,b]\rightarrow M$ is a map 
 \[D^v_\gamma: \mathfrak{X}(\gamma)\rightarrow T_{\pi(v)}M,\quad\quad X\mapsto D_\gamma^vX \]
for every $v\in A$ with $\pi(v)=\gamma(t_0)$, and $t_0\in [a,b]$, such that 
 \begin{enumerate}[(i)]
 %\item $\nabla^v_wX$ is linear in $w$,
 \item $D^v_\gamma(X+Y)=D^v_\gamma X+D^v_\gamma Y$, $X,Y\in  {\mathfrak X}(\gamma)$,
\item $D_\gamma^v(fX)=\frac{df}{dt}(t_0) X(t_0)+f(t_0) D_\gamma^vX $ $\forall \,\,f\in {\mathcal F}(I)$, $X\in  {\mathfrak X}(\gamma)$,
\item $D_\gamma^VX(t):=D_\gamma^{V(t)}X$ is smooth  $\forall\,\, V,X\in   {\mathfrak X}(\gamma)$ and $V$, $A$-admissible, namely, $V(t)\in A\quad$ $\forall\,\, t\in [a,b]$.
 \end{enumerate}
 \end{defi}
 \begin{prop}
Given a smooth curve $\gamma:[a,b]\rightarrow M$,  an $A$-anisotropic connection $\nabla$ determines an induced $A$-anisotropic covariant derivative along $\gamma$ with the following property: if $X\in {\mathfrak X}(M)$, then $D_\gamma^v (X_{\gamma})=\nabla^v_{\dot\gamma} X$, where $X_\gamma$ is the vector field in ${\mathfrak X}(\gamma)$ defined as $X_\gamma(t)=X_{\gamma(t)}$ $\forall t\in [a,b]$.  
\end{prop}
\begin{proof} Analogous to \cite[Prop. 3.18]{Oneill}. See also \cite[Prop. 18]{Ja19}.
%It follows as in \cite[Proposition 3.18]{Oneill}. Observe that  in coordinates around $\gamma(t_0)=\pi(v)$, with $t_0\in [a,b]$, the induced covariant connection is given by
% \begin{equation}
% D_\gamma^v X=\dot X^i(t_0)\frac{\partial }{\partial x^i}+X^i(t_0)\dot\gamma^j(t_0) \Gamma^k_{\, ij}(v) \frac{\partial}{\partial x^k}.
 %\end{equation}
\end{proof}
\begin{defi}
We say that a smooth curve $\gamma:[a,b]\rightarrow M$ is $A$-admissible if $\dot\gamma(t)\in A$ for all $t\in [a,b]$. Moreover, we say that an $A$-admissible smooth curve is an {\em autoparallel curve} of the $A$-anisotropic connection $\nabla$ if $D^{\dot\gamma}_\gamma \dot\gamma=0$, where $D_\gamma$ is the $A$-anisotropic covariant derivative associated with $\nabla$.
\end{defi}
In coordinates, autoparallel curves are given by the equation
\begin{equation}
\ddot \gamma^k+\dot\gamma^i\dot\gamma^j \Gamma^{k}_{\, ij}(\dot\gamma)=0.
\end{equation}
%\begin{rem}\label{differenceten}
%Every anisotropic linear connection $\nabla$ determines a spray, which is given by the tangent vector to the curves $\dot\gamma:[a,b]\rightarrow A$, where $\gamma$ is a geodesic of $\nabla$. The coefficients of the spray in a natural coordinate system $(T\Omega,\tilde{\varphi})$ are given by $G^i(v)=v^jv^k \Gamma^i_{\, jk}(v)$ for any $v\in T\Omega\cap A$. Observe that there can be several anisotropic connections that determine the same spray. In particular, $\nabla$ and $\tilde{\nabla}$ determines the same spray when the difference tensor, defined as
%\begin{equation}\label{diff}
%Q_v(X,Y)=\nabla^v_XY-\tilde{\nabla}^v_XY,
%\end{equation}
%for any $X,Y\in {\mathfrak{X}}(M)$, satisfies that $Q_v(v,v)=0$ for every $v\in A$.
%
%\end{rem}

 We say that a {\em  two-parameter map}  is a smooth map
$\Lambda:{\mathcal O}\rightarrow M$ such that  ${\mathcal O}$ is an open subset of $\R^2$ satisfying the interval condition, namely, horizontal and vertical lines of $\R^2$ intersect $\mathcal O$ on intervals. We will use the following notation:
\begin{enumerate}
\item the $t$-parameter curve of $\Lambda$
at $s_0$ is the curve $\gamma_{s_0}$ defined as $t\rightarrow \gamma_{s_0}(t)=\Lambda(t,s_0)$,
\item the $s$-parameter curve of $\Lambda$ at $t_0$ is the curve $\beta_{t_0}$ defined as $s\rightarrow \beta_{t_0}(s)=\Lambda(t_0,s)$.
\end{enumerate}
 Let us define $\Lambda^*TM$ as the pull-back vector bundle over ${\mathcal O}$ induced by lifting $\pi:TM\rightarrow M$ through $\Lambda$. Then we denote the subset of smooth sections 
of $\Lambda^*TM$ as ${\mathfrak X}(\Lambda)$:
\begin{displaymath}
    \xymatrix{ \Lambda^*TM \ar[d]_{\Lambda^*} & TM \ar[d]^{\pi} \\
               {\mathcal O} \ar[r]^{\Lambda}  & M  }
\end{displaymath}
 Observe that a vector field $V\in {\mathfrak X}(\Lambda)$ induces vector fields in $\mathfrak X(\gamma_{s_0})$ and $\mathfrak X(\beta_{t_0})$.
We can also define the {\it   curvature operator}  associated with an $A$-admissible two-parameter map $\Lambda: [a,b]\times (-\varepsilon,\varepsilon)\rightarrow M$, $(t,s)\rightarrow \Lambda(t,s)$.  Here $A$-admissible means that $\dot\gamma_s(t)\in A$ for every $(t,s)\in [a,b]\times (-\varepsilon,\varepsilon)$. The curvature operator of $\Lambda$ is a map
$R_\Lambda: {\mathfrak X}(\Lambda)\rightarrow {\mathfrak X}(\Lambda)$ defined, for
 any vector field $W\in {\mathfrak X}(\Lambda)$, as
 \[R_\Lambda(W):=D_{\gamma_{s}}^{\dot\gamma_s}D_{\beta_{t}}^{\dot\gamma_s}W-D_{\beta_{t}}^{\dot\gamma_s}D_{\gamma_{s}}^{\dot\gamma_s}W-P_{\dot\gamma_s}(\dot\beta_t,W,D_{\gamma_s}^{\dot\gamma_s}\dot\gamma_s)+
 P_{\dot\gamma_s}(\dot\gamma_s,W,D_{\beta_t}^{\dot\gamma_s}\dot\gamma_s).\]
% {\bb PROBAR: \eb}
% 
%  As in \cite[Theorem 1.1]{Ja14b}, when the covariant derivative is induced by an anisotropic linear connection, it is possible to prove that 

\begin{prop}
	Given a two-parameter map, and an $A$-anisotropic connection in a manifold $M$, with $R_\Lambda$ the curvature operator of it induced covariant derivative, it holds
\begin{equation}\label{RLambda}
R_\Lambda (W)=R_{\dot\gamma_s}(\dot\gamma_s,\dot\beta_t)W,
\end{equation}
where $R_{\dot\gamma_s}$ is the curvature tensor of $\nabla$. 
\end{prop}
\begin{proof}
	First observe that by a straightforward computation, one can check that $R_\Lambda$ is ${\mathcal F}(I)$-multilinear on $W$, namely, given $f\in {\mathcal F}(I)$, $R_\Lambda(fW)=f R_\Lambda(W)$. Then in order  to check \eqref{RLambda} is enough to prove that $R_\Lambda (\left(\partial_i\right)_{\gamma_s})=R_{\dot\gamma_s}(\dot\gamma_s,\dot\beta_t)\partial_i$ for any partial vector field $\partial_i$. This is also straightforward taking into account that 
	\begin{align*}
	\frac{d}{dt}\left(\Gamma^k_{\,\, ij}\circ \frac{d\Lambda}{dt}\right)=\frac{d\Lambda^m}{dt}\frac{\partial \Gamma^k_{\,\, ij}}{\partial x^m}\circ \frac{d\Lambda}{dt}+ \frac{d^2\Lambda^m}{dt^2}\frac{\partial \Gamma^k_{\,\, ij}}{\partial y^m}\circ \frac{d\Lambda}{dt},\\
	\frac{d}{dt}\left(\Gamma^k_{\,\, ij}\circ \frac{d\Lambda}{dt}\right)=\frac{d\Lambda^m}{ds}\frac{\partial \Gamma^k_{\,\, ij}}{\partial x^m}\circ \frac{d\Lambda}{dt}+ \frac{d^2\Lambda^m}{dsdt}\frac{\partial \Gamma^k_{\,\, ij}}{\partial y^m}\circ\frac{d\Lambda}{dt},
	\end{align*}
parts $(i)$ and $(ii)$ of Def. \ref{deficov} and \eqref{Firstcur2}.
	\end{proof}
%Observe that as we do not assume $\nabla$ to be torsion-free, in principle, $D_{\beta_t}^{\dot\gamma_s}\dot\gamma_s\not=D_{\gamma_s}^{\dot\gamma_s}\dot\beta_t$, and then, the above definition for $R_\Lambda$ does not coincide exactly with the one given in \cite[Theorem 1.1]{Ja14b}. Let us characterize geodesic variations with the curvature operator.
\begin{defi}
Given an auto-parallel curve $\gamma$ of an $A$-anisotropic connection $\nabla$, we say that a vector field $J$ along $\gamma$ is a {\em Jacobi field} if it is the variational vector field of a variation of $\gamma$  such that the longitudinal curves (namely, in the notation above, the curves $\gamma_s$) are auto-parallel curves.
\end{defi}
\begin{prop}\label{Jacoboper}
Let $\nabla$ be an $A$-anisotropic connection in $A\subset TM\setminus 0$, being $\mathcal T$, $P$ and $R$, respectively, the torsion, the vertical derivative and the curvature tensor of $\nabla$. If $\gamma:[a,b]\rightarrow M$ is an auto-parallel curve of $\nabla$, $D_\gamma$, the induced covariant derivative along $\gamma$ and $J$, a Jacobi field along $\gamma$, then
\begin{equation}\label{Jacobieq1}
(D^{\dot\gamma}_{\gamma})^2J=R_{\dot\gamma}(\dot\gamma,J)\dot\gamma-
P_{\dot\gamma}(\dot\gamma,\dot\gamma,D^{\dot\gamma}_{\gamma}J+{\mathcal T}_{\dot\gamma}(J,\dot\gamma))-(\nabla_{\dot\gamma}{\mathcal T})_{\dot\gamma}(J,\dot\gamma)-{\mathcal T}_{\dot\gamma}(D^{\dot\gamma}_{\gamma}J,\dot\gamma).
\end{equation}
In particular, if $\nabla$ is torsion-free and
\begin{equation}\label{vertprop}
\text{$P_v(v,v,u)=0\,\,$  $\forall v\in A$ and $u\in T_{\pi(v)}M$},
\end{equation}
 then 
\begin{equation}\label{Jacobieq}
(D^{\dot\gamma}_\gamma)^2J=R_{\dot\gamma}(\dot\gamma,J)\dot\gamma.
\end{equation}
\end{prop}
\begin{proof}
%It can be proved adapting the proof of \cite[Prop. 8.1]{HJP15}.
Consider a variation $\Lambda:[a,b]\times (-\varepsilon,\varepsilon)\rightarrow M$ of $\gamma$ (with the above notation) in such a way that $\gamma_s$ is an auto-parallel curve for every $s\in  (-\varepsilon,\varepsilon)$
and $\dot\beta_t(0)=J(t)$. Then $D^{\dot\gamma_s}_{\gamma_s}\dot\gamma_s=0$ and from the definition of $R_\Lambda$ and \eqref{RLambda}, we get
\begin{equation}\label{nabla1}
0=D^{\dot\gamma_s}_{\beta_t} D^{\dot\gamma_s}_{\gamma_s}\dot\gamma_s= -R_{\dot\gamma_s}(\dot\gamma_s,\dot\beta_t)\dot\gamma_s
+D^{\dot\gamma_s}_{\gamma_s}D^{\dot\gamma_s}_{\beta_t}\dot\gamma_s
+P_{\dot\gamma}(\dot\gamma,\dot\gamma,D^{\gamma_s}_{\beta_t}\dot\gamma_s).
\end{equation}
Moreover, taking into account the definition of the torsion $\mathcal T$,
% and the property (iii) of the induced $A$-anisotropic covariant derivative, 
we get that $D^{\dot\gamma_s}_{\beta_t}\dot\gamma_s=
D^{\dot\gamma_s}_{\gamma_s}\dot\beta_t+{\mathcal T}_{\dot\gamma_s}(\dot\beta_t,\dot\gamma_s)$ and then
\begin{equation}\label{nabla2}
D^{\dot\gamma_s}_{\gamma_s}D^{\dot\gamma_s}_{\beta_t}\dot\gamma_s=
(D^{\dot\gamma_s}_{\gamma_s})^2\dot\beta_t+
D^{\dot\gamma_s}_{\gamma_s}({\mathcal T}_{\dot\gamma_s}(\dot\beta_t,\dot\gamma_s)).
\end{equation}
Furthermore,
\begin{equation}\label{nabla3}
D^{\dot\gamma_s}_{\gamma_s}({\mathcal T}_{\dot\gamma_s}(\dot\beta_t,\dot\gamma_s))=(\nabla_{\dot\gamma_s}{\mathcal T})_{\dot\gamma_s}(\dot\beta_t,\dot\gamma_s)+{\mathcal T}_{\dot\gamma_s}(D^{\dot\gamma_s}_{\gamma_s}\dot\beta_t,\dot\gamma_s)
+({\partial^\nu}{\mathcal T})_{\dot\gamma_s}(\dot\beta_t,\dot\gamma_s,D^{\dot\gamma_s}_{\dot\gamma_s}\dot\gamma_s)%+(\partial^\nu {\mathcal T})_{\dot\gamma_s}(\dot\beta_t,\dot\gamma_s,D^{\dot\gamma_s}_{\gamma_s}\dot\gamma_s.
\end{equation}
(recall \eqref{productrule} and \eqref{partialT}).
Putting together \eqref{nabla1}-\eqref{nabla3}, evaluating in $s=0$ and taking into account that $\gamma_s$ is an auto-parallel curve, we easily conclude \eqref{Jacobieq1}.
\end{proof}
\begin{defi}\label{def:curvoper}
Let $\nabla$ be an $A$-anisotropic connection in $A\subset TM\setminus 0$ %satisfying condition  \eqref{vertprop}
 and $\gamma:[a,b]\rightarrow M$ an auto-parallel curve of $\nabla$. 
 We say that the map 
\begin{equation}
R_\gamma:\mathfrak{X}(\gamma)\rightarrow \mathfrak{X}(\gamma), \quad
U\rightarrow R_\gamma(U):= R_{\dot\gamma}(\dot\gamma,U)\dot\gamma
\end{equation}
 is {\it the curvature operator of $\gamma$}.
\end{defi}
\subsection{Parallel Transport}\label{parallel} Given an $A$-anisotropic connection $\nabla$, there are several ways to transport a vector field along a curve $\gamma:[a,b]\rightarrow M$ considering the covariant derivative $D_\gamma$ associated with $\nabla$:
\begin{enumerate}[(i)]
\item The parallel transport  defined by $D_\gamma^XX=0$. If the subset $A$ does not coincide with $TM\setminus 0$, this parallel transport could not be defined along the whole curve, but at least it is defined in an interval of $a$.
\item The $\gamma$-parallel transport defined by $D_\gamma^{\dot\gamma}X=0$, which is always defined along $\gamma$ whenever $\gamma$ is $A$-admissible, namely, $\dot\gamma(t)\in A$ for every $t\in [a,b]$.
\item The $W$-parallel transport defined by $D_\gamma^{W}X=0$, which is always defined along $\gamma$ whenever $W$ is $A$-admissible.
\end{enumerate}
%By ODE Theory, it follows that both $\gamma$-parallel and $W$-parallel transport are always defined along the whole curve $\gamma$. This is because the coordinates  $X^i$ of $X$ in some coordinate system have to satisfy, respectively, the equations
Observe that in order to prove that both $\gamma$-parallel and $W$-parallel transports are always defined along the whole curve $\gamma$, it is enough to apply standard ODE Theory to  the equations
\begin{equation}
\dot X^i+\Gamma_{\,\, jk}^i(W) \dot\gamma^j X^k=0, \quad 
\dot X^i+\Gamma_{\,\, jk}^i(\dot\gamma) \dot\gamma^j X^k=0,
\end{equation}
with $i=1,\ldots, n$.
Instead, the parallel transport is not necessarily defined in the whole domain of $\gamma$, but at least it is defined in some subinterval, as this time the equations
\begin{equation}
\dot X^i+\Gamma_{\,\, jk}^i(X) \dot\gamma^j X^k=0,
\end{equation}
with $i=1,\ldots, n$, are not linear.
% In some particular cases, it is possible to ensure that the parallel transport is defined along whole domain of $\gamma$, as for example, when the anisotropic connection is positive homogeneous of degree zero, namely, $\nabla^{\lambda v}=\nabla^v$ for every $\lambda>0$ and $v\in A$ and $A=TM\setminus 0$.
%\begin{prop}
%Given an anisotropic connection $\nabla$, a smooth curve $\gamma:[a,b]\rightarrow M$ and an $A$-admissible vector field $W\in {\mathfrak X}(\gamma)$, the $\gamma$-parallel transport (when the curve $\gamma$ is $A$-admissible) and the $W$-parallel transport are always defined in $[a,b]$. Moreover, if $A=TM\setminus 0$ and $\nabla$ is positive homogeneous of degree zero, the parallel transport of any non-zero vector is also well-defined in $[a,b]$.
%\end{prop}

Recall that one can  compute the curvature tensor or the derivation of any tensor with an $A$-anisotropic connection in terms of an affine connection $\nabla^V$ using an arbitrary $A$-admissible extension $V$ of $v$ (see \eqref{partialT} and  \eqref{RvRV}). Let us show that one can always choose a suitable $V\in {\mathfrak X}(\Omega)$ (in some open subset $\Omega\subset M$) to simplify computations.
\begin{prop}\label{assump}
Given an $A$-anisotropic connection $\nabla$ and a vector $v\in A$, we can always choose an $A$-admissible extension $V$ defined in an open subset $\Omega\subset M$, such that  
 \begin{equation}\label{magiccond}
\nabla^v_{X}V=0
\end{equation}
for any vector field $X\in {\mathfrak X}(\Omega)$. Furthermore, if
$T\in  \mathfrak{T}^r_s(M,A)$ and $X\in \mathfrak{T}^1_0(M,A)$, then $(\nabla_X T)_v=(\nabla^V_X(T_V))_{\pi(v)}$, and the curvature tensor of $\nabla$ can be computed as
\begin{equation}\label{curvaturecomp}
 R_v(X,Y)Z=R^V_{\pi(v)}(X,Y)Z=(\nabla^V_{X}\nabla^V_{Y}Z)_{\pi(v)}-
 (\nabla^V_{Y}\nabla^V_{X}Z))_{\pi(v)},
 \end{equation}
 for $X,Y,Z\in  {\mathfrak X}(\Omega)$ such that $[X,Y]=0$ (the last condition is not necessary for the first identity), and its derivative as
 \begin{multline}\label{dercurvaturecomp}
 (\nabla_XR)_v(Y,Z)W=(\nabla^V_XR^V)_{\pi(v)}(Y,Z)W-P_v(Z,W,\nabla_X^V\nabla^V_YV)\\+P_v(Y,W,\nabla_X^V\nabla^V_ZV)
 \end{multline}
 assuming that all the  Lie brackets of the vector fields $X,Y,Z,W\in  {\mathfrak X}(\Omega)$ are zero.
\end{prop}

\begin{proof}
 In order to find the extension,  choose a chart $(\Omega,\varphi)$ in such a way that $\varphi(\Omega)$ is a product of open intervals. Then extend $v$ along the integral curves of the chart using the parallel transport given by $D_\gamma^VV=0$, namely, if $\varphi(p)=(x^1,\ldots,x^n)$, first extend $v$ to a parallel vector field along the curve $(a_1,b_1)\ni t\rightarrow \varphi^{-1}(t,x^2,\ldots,x^n)$, then to a parallel vector field along $(a_2,b_2)\ni t\rightarrow \varphi^{-1}(s,t,x^3,\ldots,x^n)$, for every $s\in (a_1,b_1)$ and so on, obtaining a vector field in $\Omega$. Observe that as the parallel transport is not defined in all the interval, we may need to reduce $\Omega$. The identity for $\nabla_XT$ follows directly from \eqref{productrule}, \eqref{partialT} and \eqref{magiccond}.  The identity \eqref{curvaturecomp} follows from \eqref{RvRV} and \eqref{magiccond}, and for \eqref{dercurvaturecomp}, use the identity  $(\nabla_X T)_v=(\nabla^V_X(T_V))_{\pi(v)}$ for $T=R$ and observe that 
 \[\left(\nabla^V_X(P_V(Z,W,\nabla^V_YV))\right)_{\pi(v)}=P_v(Z,W,\nabla_X^V\nabla^V_YV)\]
  as a consequence of \eqref{magiccond}.
\end{proof}
Observe that with the choice of $V$ in \eqref{magiccond}, one has that $(\nabla^V_{X}V)_p=0$, where $p=\pi(v)$, but the vector field $\nabla^V_{X}V$ could not be identically zero away from $p$.
\subsection{Bianchi identities}\label{bianchiid} Let us generalize Bianchi identities to arbitrary anisotropic connections.
\begin{prop}
Let $\nabla$ be an $A$-anisotropic connection and $P$, $\mathcal T$ and $R$ its vertical derivative, and torsion  and curvature tensors, respectively. For every $v\in A$ and $u,w,z\in T_{\pi(v)}M$, we have that $R_v(u,w)=-R_v(w,u)$ and $R$ satisfies the first Bianchi identity:
\[\sum_{cyc: u,w,z} R_v(u,w)z=\sum_{cyc:u,w,z} ({\mathcal T}_v({\mathcal T}_v(u,w),z)+(\nabla_u{\mathcal T})_v(w,z)),\]
and the second Bianchi identity:
\[\sum_{cyc:u,w,z} \big((\nabla_uR)_v(w,z)b-P_v(w,b,R_v(u,z)v)+R_v({\mathcal T}_v(u,w),z)b\big)=0.\]
Here $\sum_{cyc: u,w,z}$ denotes the cyclic sum in $u,w,z$.
\end{prop}

\begin{proof}
Consider extensions $X,Y,Z,W$  of $u,w,z,b$ respectively  in such a way that its Lie brackets are zero  and an extension $V$ of $v$ satisfying \eqref{magiccond}.
 Recall that $R^V$ satisfies the  Bianchi Identities (see for example \cite[Th. 5.3]{KN63}). Moreover,  observe that with our choice of $V$,
% $R^V_{\pi(v)}= R_v$, 
$\nabla^V_X({\mathcal T}_V)_{\pi(v)}=(\nabla_X{\mathcal T})_v$ (recall Prop. \ref{assump}) and it holds \eqref{curvaturecomp} and \eqref{dercurvaturecomp}.
% \begin{align*}
% (\nabla_X R)_v(Y,Z)W=&(\nabla^V_XR^V)_{\pi(v)}(Y,Z)W-X_{\pi(v)}(P_V(Z,W,\nabla^V_YV))\\&+X_{\pi(v)}(P_V(Y,W,\nabla^V_ZV))\\
% =&(\nabla^V_XR^V)_{\pi(v)}(Y,Z)W-P_v(Z,W,\nabla^V_X\nabla^V_YV)\\&+P_v(Y,W,\nabla^V_X\nabla^V_ZV).
% \end{align*}
 Making the cyclic sum, one easily concludes the second Bianchi identity.
\end{proof}
Finally, we will give the vertical Bianchi identity.
\begin{prop}
Let $\nabla$ be an $A$-anisotropic connection and $P$, $\mathcal T$ and $R$ its vertical derivative and torsion  and curvature tensors respectively. For every $v\in A$ and $u,w,z,b\in T_{\pi(v)}M$,
\begin{multline}\label{verticalBianchi}
(\partial^\nu R)_v(u,w,z,b)=(\nabla_uP)_v(w,z,b)-(\nabla_wP)_v(u,z,b)
+P_v({\mathcal T}_v(u,w),z,b)\\-P_v(w,z,P_v(u,v,b))+P_v(u,z,P_v(w,v,b)).
\end{multline}
\end{prop}
\begin{proof}
Let $V,X,Y,Z,W$ be vector fields extensions of $v,u,w,z,b$ with $V$, $A$-admissible satisfying  \eqref{magiccond}, and such that the Lie brackets of $X,W,Z,Y$ cancel. Then
\begin{align*}
\left.\frac{d}{dt}\right|_{t=0}(\nabla^{V+tW}_X\nabla^{V+tW}_YZ)_{\pi(v)}=&\nabla^v_X(P_V(Y,Z,W))
+P_v(X,\nabla^V_YZ,W)\\
=&(\nabla_XP)_v(Y,Z,W)+P_v(\nabla^V_XY,Z,W)\\&+P_v(Y,\nabla_X^VZ,W)
%&\quad\quad\quad
 +P_v(Y,Z,\nabla^V_XW)\\&+P_v(X,\nabla_Y^VZ,W)
\end{align*}
and 
\begin{multline*}
\left.\frac{d}{dt}\right|_{t=0}P_{v+tw}(Y,Z,\nabla^{V+tW}_X(V+tW))=P_v(Y,Z,P_V(X,V,W))\\+P_v(Y,Z,\nabla^V_XW).
\end{multline*}
Taking into account  the above identities together with the ones obtained by interchanging $X$ and $Y$ in those identities and replacing the four identities  in the definition  of $\partial^\nu R$, after much cancellation, one concludes \eqref{verticalBianchi}.
\end{proof}

\subsection{Comparison of the curvature tensors}\label{curvdiff} Observe that given two different $A$-anisotropic connections $\hat\nabla$ and $\nabla$ defined in the same open subset $A\subset TM$, their difference is an $A$-anisotropic tensor defined as
\begin{equation}\label{differencetensor}
Q_v(X,Y)=\hat\nabla^v_XY-\nabla^v_XY,
\end{equation}
for any $X,Y\in {\mathfrak X}(M)$. % (recall that the $A$-anisotropic connections can be extended to ${\mathfrak T}^1_0(M,A)\times {\mathfrak T}^1_0(M,A)$, see Remark \ref{extensions}).
%(recall the comments around formulas \eqref{forthecurv} and \eqref{interpr}, Remark \ref{Tensorequiv} and \eqref{partialTexpres}).
Let us relate the curvature tensors of both connections.
\begin{prop}
Let $\hat{R}$, $R$ be the curvature tensors associated with $\hat\nabla$ and $\nabla$ respectively, and $u,w,z\in T_{\pi(v)}M$. Then
\begin{align}\label{RoftildeR}
\hat{R}_v(u,w)z=&R_v(u,w)z-P_v(w,z,Q_v(u,v))+P_v(u,z,Q_v(w,v))
+Q'_v(u,w)z
\end{align}
where
\begin{align}
Q'_v(u,w)z=& (\nabla_uQ)_v(w,z)-(\nabla_wQ)_v(u,z)
+Q_v({\mathcal T}(u,w),z)\nonumber\\&+\partial^\nu Q_v(u,z,Q_v(w,v))-\partial^\nu Q_v(w,z,Q_v(u,v))\\&
+Q_v(u,Q_v(w,z))-Q_v(w,Q_v(u,z)).\label{Qprima}
\end{align}
\end{prop}
\begin{proof}
Let $V,X,Y,Z$ be local extensions of $v,u,w,z$, respectively, being $V$, $A$-admissible. We can assume that $[X,Y]=0$, and $V$ satisfies \eqref{magiccond} for the connection $\hat\nabla$. Then $\nabla^v_XV=-Q_v(X,V)$ and $\nabla^v_YV=-Q_v(Y,V)$. It follows that
\begin{align}\label{curvature}
\hat R_v(X,Y)Z=(\hat\nabla^V_X\hat\nabla^V_YZ
-\hat\nabla^V_Y\hat\nabla^V_XZ)_{\pi(v)}.
\end{align}
Moreover, 
\begin{align}
\hat\nabla^V_X\hat\nabla^V_YZ=&\hat\nabla^V_X(\nabla^V_YZ+Q_V(Y,Z))\nonumber\\
=& \nabla^V_X \nabla^V_YZ+Q_V(X,\nabla^V_YZ)+\nabla^V_X(Q_V(Y,Z))+Q_V(X,Q_V(Y,Z)).\label{firstpartR}
\end{align}
Analogously,
\begin{align}\label{secondpartR}
\hat\nabla^V_Y\hat\nabla^V_XZ= \nabla^V_Y \nabla^V_XZ+Q_V(Y,\nabla^V_XZ)+\nabla^V_Y(Q_V(X,Z))+Q_V(Y,Q_V(X,Z)).
\end{align}
We also have that
\begin{multline}\label{derivQ}
(\nabla^V_X(Q_V(Y,Z)))_{\pi(v)}=(\nabla_XQ)_v(Y,Z)+Q_v(\nabla^V_XY,Z)+Q_v(Y,\nabla^V_XZ)\\+\partial^\nu Q_v(Y,Z,-Q_V(X,V)),
\end{multline}
\begin{multline}\label{derivQ2}
(\nabla^V_Y(Q_V(X,Z)))_{\pi(v)}=(\nabla_YQ)_v(X,Z)+Q_v(\nabla^V_YX,Z)+Q_v(X,\nabla^V_YZ)\\+\partial^\nu Q_v(X,Z,-Q_V(Y,V)),
\end{multline}
and 
\begin{multline}\label{secCurvature}
 R_v(X,Y)Z=(\nabla^V_X\nabla^V_YZ
-\nabla^V_Y\nabla^V_XZ)_{\pi(v)}\\ -P_v(Y,Z,-Q_V(X,V))+P_v(X,Z,-Q_V(Y,V)).
\end{multline}
Using successively in \eqref{curvature}, the identities \eqref{firstpartR}-%, \eqref{secondpartR}, \eqref{derivQ}, \eqref{derivQ2} and 
\eqref{secCurvature}, we finally get \eqref{RoftildeR}.
\end{proof}
\begin{cor}\label{Qv0}
Given two $A$-anisotropic connections with a difference tensor $Q_v$ which satisfies $Q_v(u,v)=0$ for every $v\in A$ and $u\in T_{\pi(v)}M$, then
\begin{multline}\label{Qvuv}
\hat{R}_v(u,w)z=R_v(u,w)z+(\nabla_uQ)_v(w,z)-(\nabla_wQ)_v(u,z)+Q_v({\mathcal T}(u,w),z)\\+Q_v(u,Q_v(w,z))-Q_v(w,Q_v(u,z))
\end{multline}
and $\hat{R}_v(u,w)v=R_v(u,w)v$ for any $v\in A$ and $u,w,z\in T_{\pi(v)}M$.
\end{cor}
\begin{proof}
The identity \eqref{Qvuv} follows straightforwardly from \eqref{RoftildeR}. For the identity $\hat{R}_v(u,w)v=R_v(u,w)v$ we only need to use \eqref{Qvuv} observing that 
$(\nabla_uQ)_v(w,v)=0$. In order to check this, consider local extensions $X,Y,V$ of $u,w,v$, respectively, with $V$, $A$-admissible satisfying \eqref{magiccond}, and apply definitions.
\end{proof}
\begin{prop}\label{diffJacobi}
Let $\nabla$ be a
 torsion-free $A$-anisotropic connection with vertical derivative satisfying \eqref{vertprop}, and $\hat{\nabla}$ any other torsion-free $A$-anisotropic connection  with difference tensor \eqref{differencetensor} with respect to $\nabla$ satisfying
\begin{equation}\label{cond}
\text{$Q_v(v,u)=0,\,\,$
$\forall v\in A$ and $u\in T_{\pi(v)}M$}.
\end{equation}
Then 
\begin{enumerate}[(i)]
\item for every $v\in A$ and $u,w\in T_{\pi(v)}M$, the vertical derivative of $Q$ satisfies that
\begin{equation}\label{partialL}
\partial^\nu Q_v(v,u,w)=-Q_v(w,u),
\end{equation}
\item the  vertical derivative of $\hat{\nabla}$ satisfies \eqref{vertprop},  
\item  $\hat\nabla$ has the same curvature operator  (recall Def. \ref{def:curvoper}) and the same Jacobi equation \eqref{Jacobieq} as $\nabla$.
\end{enumerate}
\end{prop}
\begin{proof}
Observe that for any vector field $J$ along $\gamma$, 
$(\hat{D}^{\dot\gamma}_{\gamma})^2 J=(D^{\dot\gamma}_{\gamma})^2J$ because of condition \eqref{cond}, where $\hat D_\gamma$ and $D_\gamma$ are the  $A$-anisotropic covariant derivatives along $\gamma$ induced by $\hat\nabla$ and $\nabla$, respectively.  Moreover, using  that $Q_{v+tw}(v+tw,u)=0$ for every $t\in\R$ and computing the derivative with respect to $t$, we get \eqref{partialL}. If $\hat{P}$ is the vertical derivative of $\hat{\nabla}$, then
using \eqref{partialL} and \eqref{cond}, we get
\[\hat{P}_v(v,v,u)=P_v(v,v,u)+\partial^\nu Q_v(v,v,u)=-Q_v(u,v)=-Q_v(v,u)=0.\]
Here, we have also used that $Q$ is symmetric, because $\nabla$ and $\hat \nabla$ are both torsion-free.
 This implies that  the Jacobi equation for $\hat\nabla$ is of the form \eqref{Jacobieq}, and using the last statement of Cor. \ref{Qv0}, we conclude $(iii)$. %then as the Jacobi fields are the same for both connections and we have already observed that $(\tilde{\nabla}^{\dot\gamma}_{\gamma})^2 J=(\nabla^{\dot\gamma}_{\gamma})^2J$, we conclude that both connections have the same curvature operator.
\end{proof}
\section{Distinguished connections}\label{goodcon}
In this section, we will study a family of $A$-anisotropic connections which are suitable to study the geometry of pseudo-Finsler metrics.
Let $A\subset TM\setminus \bf 0$ be an open  conic subset, namely, an open subset of $TM$ satisfying that for every $v\in A$ and $\lambda>0$ we have that $\lambda v\in A$. We define a pseudo-Finsler metric on $A$ as a smooth, positive two-homogeneous function $L:A\subset TM\setminus 0\rightarrow \R$,  such that its  fundamental tensor defined as 
\begin{equation}\label{fundten}
g_v(u,w):=\frac 12 \frac{\partial^2}{\partial t\partial s}L(v+tu+sw)|_{t=s=0}
\end{equation}
for every $v\in A$ and $u,w\in T_{\pi(v)}M$, is non-degenerate. The {\em Cartan tensor} associated with $L$ is defined as
\begin{equation}
C_v(w_1,w_2,w_3):=\left.\frac 14 \frac{\partial^3}{\partial s_3\partial s_2\partial s_1}L\left(v+\sum_{i=1}^3s_iw_i\right)\right|_{
s_1=s_2=s_3=0}.
\end{equation}
Recall that $C_v$ is symmetric and, by homogeneity, one has that $C_v(v,u,w)=C_v(u,v,w)=C_v(u,w,v)=0$ for any $v\in A$ and $u,w\in T_{\pi(v)}M$.  In this context, it is possible to define a Levi-Civita $A$-anisotropic connection, namely, a torsion-free $A$-anisotropic connection  $\nabla$ such that $\nabla g=0$, where $g$ is the fundamental tensor. This connection can be identified with the Chern connection (see \cite[Eqs. (7.20) and (7.21)]{Sh01} and \cite[\S 4.1 and \S 4.4]{Ja19}), so we will refer to it sometimes as the Levi-Civita-Chern connection.
 Moreover, the curvature tensor of this connection has some symmetric properties with respect to the fundamental tensor of the pseudo-Finsler metric. These symmetries can also be  found  in \cite[\S 3.4A]{BCS00}.
\begin{prop}
Let $(M,L)$ be a pseudo-Finsler manifold and $\nabla$, its Levi-Civita-Chern connection. Then  the curvature tensor $R$ associated with $\nabla$ satisfies the symmetries:
\begin{align}\label{symR}
g_v(R_v(u,w)z,b)+g_v(R_v(u,w)b,z)=2C_v(R_v(w,u)v,z,b)
\end{align}
and 
\begin{multline}\label{seisB}
g_v(R_v(u,w)z,b)-g_v(R_v(z,b)u,w)=\\
C_v(R_v(w,z)v,u,b)+C_v(R_v(z,u)v,w,b)+C_v(R_v(u,b)v,z,w)\\
+C_v(R_v(b,w)v,z,u)+C_v(R_v(z,b)v,u,w)+C_v(R_v(w,u)v,z,b).
\end{multline}
%\begin{multline}
% B^V(Z,Y,X,W)+B^V(X,Z,Y,W)+B^V(W,X,Z,Y)\\
% +B^V(Y,W,Z,X)+B^V(W,Z,X,Y)+B^V(X,Y,Z,W).
%\end{multline}
\end{prop}
\begin{proof}
Let $V,X,Y,Z,W$ be local extensions of $v,u,w,z,b$, respectively, being $V$, $A$-admissible and satisfying \eqref{magiccond}. Then using \cite[Prop. 3.1]{Jav14a}, we easily conclude \eqref{symR} and \eqref{seisB}, because in this case $R_V(X,Y)Z=R^V(X,Y)Z$.
\end{proof}
\subsection{Torsion-free $A$-anisotropic connections and pseudo-Finsler metrics}\label{torsionfree} Assume that $\hat\nabla$ is a torsion-free $A$-anisotropic connection, $(M,L)$ is a pseudo-Finsler manifold as above and define $\Q$ as the $A$-anisotropic tensor
\begin{equation}\label{nablag1}
\Q_v(u,w,z)=(\hat\nabla_ug)_v(w,z)
\end{equation}
for every $v\in A$ and $u,w,z\in T_{\pi(v)}M$.
Then the $A$-anisotropic connection $\hat\nabla$ satisfies a Koszul type formula:
\begin{multline}\label{KoszulQ}
 2 g_v(\hat\nabla^V_XY,Z)= X_{\pi(v)}(g_V(Y,Z))-Z_{\pi(v)} (g_V(X,Y))+Y_{\pi(v)} (g_V(Z,X))\\
+g_v([X,Y],Z)+g_v([Z,X],Y)-g_v([Y,Z],X)\\
2( -C_v(Y,Z,\hat\nabla^V_XV)-C_v(Z,X,\hat\nabla^V_YV)+C_v(X,Y,\hat\nabla^V_ZV))\\
 -\Q_v(Y,Z,X)-\Q_v(Z,X,Y)+\Q_v(X,Y,Z),
\end{multline}
where $V$ is an $A$-admissible  local extension of $v\in A$ and $X,Y,Z$ are arbitrary vector fields. This expression can be obtained as the Koszul formula for the Chern connection using that $\hat \nabla$ is torsion-free and
\begin{multline*}
(\hat\nabla_ug)_v(w,z)=X_{\pi(v)}(g_V(Y,Z))-g_v(\hat\nabla^V_XY,Z)-g_v(Y,\hat\nabla^V_XZ)\\-2C_v(Y,Z,\hat\nabla^V_XV),
\end{multline*}
recall \cite[\S 4.1]{Ja19}.

\begin{prop}\label{torsionconnection}
Given a pseudo-Finsler manifold $(M,L)$ on a conic open subset  $A\subset TM\setminus \bf 0 $ and an $A$-anisotropic tensor $\Q$, there is a unique torsion-free $A$-anisotropic connection $\hat\nabla$  satisfying \eqref{nablag1}. Moreover, if $\Q$ is symmetric,
%and $\Q_v(v,u,w)=0$ for every $v\in A$ and $u,w\in T_{\pi(v)}M$, 
then $\hat\nabla= \nabla-\frac 1 2 \Q^\flat$, where $\nabla$ is the Chern connection of $L$ and the tensor $\Q^\flat$ is determined by $g_v(\Q^\flat_v(u,w),z)=\Q_v(u,w,z)$.
\end{prop}
\begin{proof}
For the first statement, observe that the Koszul formula  when $X=Y=V$, being $V$ an arbitrary extension of  $v$, reduces to
\begin{multline*}
2g_v(\hat\nabla^V_VV,Z)=2 v(g_V(V,Z))-Z_{\pi(v)}(g_V(V,V))+2g_v([Z,V],V)\\-\Q_v(V,Z,V)-\Q_v(Z,V,V)+\Q_v(V,V,Z),
\end{multline*}
and when $Y=V$,
\begin{multline}\label{nablaXV}
2g_v(\hat\nabla^V_XV,Z)=X_{\pi(v)} (g_V(V,Z))-Z_{\pi(v)} (g_V(X,V))+v (g_V(Z,X))\\
+g_v([X,V],Z)+g_v([Z,X],V)-g_v([V,Z],X)\\
 -2C_v(Z,X,\hat\nabla^V_VV)
-\Q_v(V,Z,X)-\Q_v(Z,X,V)+\Q_v(X,V,Z).
\end{multline}
Therefore, $\hat\nabla^V_VV$ and $\hat\nabla^V_XV$ are determined and then \eqref{KoszulQ} completely determines $\hat\nabla^V_XY$. Moreover, from \eqref{KoszulQ} and \eqref{nablaXV}, it is not difficult to prove that $\hat\nabla$ must satisfy the properties $(i)-(iii)$ in Def. \ref{aniconnection} and it is ${\mathcal F}(M)$-linear in $X$. The relation $\hat\nabla=\nabla-\frac 12 \Q^\flat$ follows easily taking into account the Koszul formulae for $\hat\nabla$ and $\nabla$.
\end{proof}
\begin{rem}
It is well-known that geodesics of a pseudo-Finsler metric are the auto-parallel curves of the Levi-Civita-Chern connection $\nabla$. Then an $A$-anisotropic connection $\hat \nabla$ as above has the same auto-parallel curves (including the parametrization) as $\nabla$ if and only if $\Q^\flat_v(v,v)=0$ for every $v\in A$.
\end{rem}
From now on, we will fix a symmetric $A$-anisotropic  tensor $\Q$ satisfying that $ \Q_v(v,u,w)=0$ for every $v\in A$ and $u,w\in T_{\pi(v)}M$ and will denote by $\hat\nabla$ the $A$-anisotropic connection $\hat\nabla$ associated with $\Q$, which is determined by
\begin{equation}\label{nablag=Q}
\hat\nabla g=\Q, %\quad,
\end{equation}
 (see Prop. \ref{torsionconnection}). Observe that by the above Remark, $\hat \nabla$ and $\nabla$ have the same auto-parallel curves,  because the property $ \Q_v(v,u,w)=0$  implies that the difference tensor $\frac{1}{2}\Q^\flat$ satisfies that $\frac 12\Q^\flat_v(v,u)=0$ for all $v\in A$ and $u\in T_{\pi(v)}M$.
%with $\Q\in \mathfrak{T}_3^0(M,A)$ symmetric and with the property $\Q_v(v,u,w)=0$ for every $v\in A$ and $u,w\in T_{\pi(v)}M$.
% and we will denote also with $\Q^\flat$ the metrically equivalent tensor defined by $g_v(\Q^\flat_v(u,w),z)=\Q_v(u,w,z)$ (recall Proposition \ref{torsionconnection}).
  Let us see that we can obtain formulas for the variations of the energy with such connections, but before we need some technical results.
 \begin{lemma}\label{lem:sym}
Let $\nabla$ be the Chern connection and $\hat\nabla$ and $\Q$ as in \eqref{nablag=Q}, with $\hat{R}$ the curvature tensor of $\hat\nabla$. Then for $v\in A$ and $u,w,z\in T_{\pi(v)}M$, one has
\begin{enumerate}[(i)]
\item $(\nabla_u \Q^\flat)_v(v,w)=(\nabla_u \Q^\flat)_v(w,v)=0$,
\item $g_v(\Q^\flat_v(u,w),v)=g_v((\nabla_u \Q^\flat)_v(w,z),v)=0$,
\item $g_v(R_v(u,w)z,v)=g_v(\hat{R}_v(u,w)z,v)$,
\item $g_v(\hat{R}_v(u,w)z,v)=-g_v(\hat{R}_v(u,w)v,z)$.
\end{enumerate}
\end{lemma}
\begin{proof}
For $(i)$ and $(ii)$ use the properties of $\Q$ and an extension $V$ of $v$ satisfying \eqref{magiccond}. In particular, for the last identity in part $(ii)$ use the almost-compatibility with the metric of the Chern connection. Part $(iii)$ is a consequence of part $(ii)$ and Cor. \ref{Qv0}. For part $(iv)$, use Cor. \ref{Qv0}, which in particular implies that  $g_v(\hat{R}_v(u,w)v,z)=g_v(R_v(u,w)v,z)$. Putting together the last identity with part $(iii)$ and  taking into account \eqref{symR}, which implies that $g_v(R_v(u,w)z,v)=-g_v(R_v(u,w)v,z)$, we conclude.
\end{proof}
Recall that the Berwald connection $\tilde\nabla$ is defined for a spray. Indeed, the Christoffel symbols of the Berwald connection are computed as the second derivatives of the coefficients of the spray. Moreover, a pseudo-Finsler metric determines a spray (see \cite{Sh01}) and then an anisotropic Berwald connection (see \cite[Def. 22]{Ja19}). The {\em Berwald tensor} $B$ is defined as the vertical derivative of $\tilde\nabla$,
%\begin{equation}\label{berwaldtensor}
%B_v(u,w,z)= \frac{\partial}{\partial t}\left(\tilde{\nabla}^{V+tZ}_XY\right)|_{t=0}
%\end{equation}
 (see (6.4) in \cite{Sh01}) and the {\em Chern tensor} $P$ as the vertical derivative of $\nabla$ (see (7.23) in \cite{Sh01}, where it has the opposite sign). As $\tilde{\nabla}$ and $\nabla$ are torsion-free, $B$ and $P$ are symmetric in the first two components, and by homogeneity, it follows that $B_v(u,w,v)=P_v(u,w,v)=0$. Furthermore, the Berwald tensor is symmetric, and then
\begin{equation}\label{Berv0}
B_v(v,u,w)=B_v(u,v,w)=B_v(u,w,v)=0.
\end{equation}
Finally,  the {\em  Landsberg curvature} of a pseudo-Finsler metric $L$ is defined as
 \begin{equation}\label{landsberg}
 \mathfrak{L}_v(u,w,z)=\frac{1}{2}g_v(B_v(u,w,z),v)
 \end{equation}
 (see (6.25) in \cite[Def. 6.2.1]{Sh01} where it has the opposite sign). From \eqref{Berv0}, it follows that 
 \begin{equation}\label{Land0}
\mathfrak{L}_v(v,u,w)=\mathfrak{L}_v(u,v,w)=\mathfrak{L}_v(u,w,v)=0.
\end{equation}
With these definitions, we can write down the difference tensor between the Chern and Berwald connections as 
\begin{equation}\label{differtensor}
 \nabla^v_XY-\tilde{\nabla}^v_XY=\mathfrak{L}^\flat_v(X,Y),
\end{equation}
for any $X,Y\in {\mathfrak X}(M)$, where  $\mathfrak{L}^\flat$ is determined by $g_v(\mathfrak{L}^\flat_v(u,w),z)=\mathfrak{L}_v(u,w,z)$ (see (7.17) in \cite{Sh01} and observe that the notation for the Chern and Berwald connections is changed). 
\begin{lemma}\label{chernvert}
Given a pseudo-Finsler metric $L$, the vertical derivative of its Chern connection satisfies \eqref{vertprop}.
\end{lemma}
\begin{proof}
Observe that the Berwald connection is torsion-free  and its vertical derivative satisfies \eqref{vertprop} (it follows from \eqref{Berv0}). Moreover, the difference tensor between the Chern connection $\nabla$ and the Berwald connection $\tilde\nabla$ is $\mathfrak{L}^\flat$ (see \eqref{differtensor}) and $\mathfrak{L}^\flat_v(v,u)=0$ for every $v\in A$ and $u\in T_{\pi(v)}M$ (it follows from \eqref{Land0}).  By part $(ii)$ of Prop. \ref{diffJacobi}, the vertical derivative of the Chern connection also satisfies \eqref{vertprop}.
\end{proof}
Recall that given  a pseudo-Finsler metric $L:A\rightarrow \R$, for every $v\in A$, one can define the {\it flag curvature} $ K_v:T_{\pi(v)}M\rightarrow \R$ 
using one of the classical linear connections. In particular, when the Chern connection is considered as  an $A$-anisotropic connection, then the flag curvature is expressed 
in terms of its associated  curvature tensor as 
\[K_v(w)=\frac{g_v(R_v(v,w)w,v)}{g_v(w,w)L(v)-g_v(v,w)^2},\]
where $g$ is the fundamental tensor of $L$ and  $w\in T_{\pi(v)}M$. One way to check this formula is by observing that the (non-null) Christoffel symbols of the Chern connection in \cite[Eq. (2.4.9)]{BCS00} as a linear connection coincide with the Christoffel symbols of the $A$-anisotropic Chern connection, and then the flag curvature in  \cite[\S 3.9A]{BCS00} coincides with the one given above (use \eqref{Firstcur2} to check this). By part $(iv)$ of Lemma \ref{lem:sym}, the flag curvature can also be obtained with any of the distinguished $A$-anisotropic connections 
$\hat\nabla$ defined in \eqref{nablag=Q}, replacing in the above formula $R$ by $\hat R$.
\begin{prop}\label{Jacobioper}
Given a pseudo-Finsler manifold $(M,L)$ on $A$, the torsion-free $A$-anisotropic connection $\hat\nabla$ satisfying \eqref{nablag=Q}  determines the same flag curvature, the same Jacobi operator  and the same Jacobi equation \eqref{Jacobieq} and its vertical derivative satisfies \eqref{vertprop} as the Levi-Civita-Chern connection. Moreover, the vertical derivative $\hat P$ of $\hat\nabla$ satisfies also that
\begin{equation}\label{gP}
g_v(\hat P_v(v,u,w),v)=0,
\end{equation}
for every $v\in A$ and $u,w\in T_{\pi(v)}M$.
\end{prop}
\begin{proof}
 By Lemma \ref{chernvert}, the vertical derivative of the Lev-Civita-Chern connection satisfies \eqref{vertprop}. Then we can apply parts $(ii)$ and $(iii)$ of Prop. \ref{diffJacobi} and part $(iii)$ of Lemma \ref{lem:sym}, which concludes all the claims except \eqref{gP}. In order to prove \eqref{gP}, observe that using \eqref{differtensor}, one gets 
 \[\hat P_v(z,u,w)=B_v(z,u,w)+(\partial^\nu\mathfrak{L}^\flat)_v(z,u,w)+(\partial^\nu\Q^\flat)_v(z,u,w).\]
 Applying part $(i)$ of Prop. \ref{diffJacobi} to $\mathfrak{L}^\flat$ and $\Q^\flat$ and using \eqref{Berv0} and \eqref{Land0}, one easily concludes \eqref{gP} from the last identity.
\end{proof}
\begin{rem}
	Let us observe that the four classical connections provide $A$-anisotropic connections which are distinguished. More precisely,
	\begin{enumerate}[(i)]
\item 	in order to define the $A$-anisotropic connection $\nabla^v_XY$ using a classical linear connection $\nabla^c$ on the vertical fiber bundle, one has to make the derivative with respect to the horizontal lift $X^{\mathcal H}$ of $X$, where the horizontal subbundle is the classical one for a Finsler metric (see \cite[Pag. 35]{BCS00}), namely, $\nabla^v_XY=\nabla^c_{X^{\mathcal H}}Y$. Here we consider $Y\in {\mathfrak X}(M)\subset {\mathcal T}_0^1(M,A)$ or $Y\equiv Y^{\mathcal V}$, being $Y^{\mathcal V}$ the vertical lift of $Y$. 
\item It turns out that the Chern and Cartan connections induce the Levi-Civita-Chern $A$-anisotropic connections, while the Hashiguchi and Ber\-wald connections give the $A$-anisotropic Berwald connection. In the case of the Chern and Berwald connections, this relation is stronger as the classical linear connections are semi-basic, namely, the derivatives with respect to vectors tangent to the vertical subbundle are zero. As a consequence, the non-null Christoffel symbols of the classical connections and its $A$-anisotropic versions coincide. For more details about the relations between derivatives and curvatures see \cite[\S 4.4]{Ja19} and for a detailed study of classical linear connections see \cite{Edu1,Edu2}.
\item It is very easy to generate a large amount of distinguished $A$-anisotropic connections from the Levi-Civita-Chern connection taking as a tensor $\Q$ combinations of the Landsberg and Cartan tensors $f{\mathcal L}+h C$, with arbitrary $f,h\in {\mathcal F}(A)$. Observe that if $f$ and $h$ are not positively homogeneous of degree zero, then the $A$-anisotropic connection will not be homogeneous of degree zero in $v$, but it is easy to see that its auto-parallel curves are the geodesics of the pseudo-Finsler metric affinely parametrized.
\item As we have seen above, this class of distinguished $A$-anisotropic connections allows us to compute the flag curvature of a pseudo-Finsler in a simple way. As we will see in the next section, they also provide suitable formulas for the first and second variation and for Jacobi fields (see Prop. \ref{Jacobioper}). It remains to be investigated which of these connections are more suitable to study certain classes of pseudo-Finsler manifolds. For example, it seems that the Berwald connection has in some sense better properties to study constant flag curvature manifolds than the Chern one. 
	\end{enumerate}
	\end{rem}
\subsection{The variations of the energy functional}\label{sub:variations}
Given a pseudo-Finsler manifold $(M,L)$ on $A$, we will denote by $C_A(M,[a,b])$ the space of $A$-admissible piecewise smooth curves
and  for any 
$A$-admissible piecewise smooth curve $\gamma:[a,b]\subset\R\rightarrow M$, let us define the energy functional as
\begin{equation}
E(\gamma)=\frac 12\int_a^b L(\dot\gamma) ds.
\end{equation}
Recall that if $\Ps$ is a  submanifold of $M$, we say that a vector  $v$ with $\pi(v)\in \Ps$ is  {\it orthogonal} to $\Ps$ if $g_v(v,w)=0$ for all $w\in T_{\pi(v)}\Ps$.  Then, a vector field $N$ along $\Ps$, namely, a smooth map $N:\Ps\rightarrow TM$, such that $\pi\circ N$ is the identity, is said to be orthogonal if $N_p$ is an orthogonal vector for every $p\in \Ps$.
We define  the {\em second fundamental form}
of $\Ps$ in the direction of the orthogonal vector field $N$ computed with the $A$-anisotropic connection $\hat\nabla$ (whenever $\Ps$ is non-degenerate with the metric $g_N$) as the tensor
$\hat S^\Ps_N:\mathfrak{X}(\Ps)\times \mathfrak{X}(\Ps)\rightarrow  \mathfrak{X}(\Ps)^\perp_N$
given by $\hat S_N^\Ps(U,W)={\rm nor}_N \hat\nabla^N_UW$, where ${\rm nor}_N$ is computed with the metric $g_N$, and $\mathfrak{X}(\Ps)^\perp_N$ is the space of $g_N$-orthogonal vector fields to $\Ps$.
\begin{prop}\label{variations}
Let $\hat\nabla$ be any torsion-free $A$-anisotropic connection satisfying \eqref{nablag=Q}, $\hat D_\gamma$, its associated covariant derivative along a piecewise smooth curve $\gamma:[a,b]\rightarrow M$ and $\Lambda$, an $A$-admissible 
piecewise smooth variation of   $\gamma$. Then we have the 
{\it first variation formula}
 \begin{equation}\label{firstvariation}
  \begin{aligned}[m]
  E'(0)&:=\left.\frac{d(E(\gamma_s))}{ds}\right|_{s=0}\\%:=(s\mapsto E(\gamma_s))'(0)\\
  &=-\int_a^bg_{\dot\gamma}(W,\hat D^{\dot\gamma}_\gamma\dot\gamma)~dt +
g_{\dot\gamma}(W,\dot\gamma)|^b_a\\
  &\quad\quad\quad +\sum_{i=1}^h\big(\mathcal L_L(\dot\gamma(t_i^+))(W(t_i))-\mathcal L_L(\dot\gamma(t_i^-))(W(t_i))\big),
  \end{aligned}
 \end{equation}
 where $\dot\gamma(t_i^+)$ (resp. $\dot\gamma(t_i^-)$), $i=1,\ldots,h$, denotes the right (resp.\ left)  velocity at the breaks $a<t_1<\ldots<t_h<b$, and $\mathcal L_L(v)(w)=g_v(v,w)$ is the Legendre transform. Moreover, if $\gamma$ is a geodesic which is orthogonal to two  submanifolds $\Ps$ and $\tilde \Ps$ at the endpoints and such that 
$g_{\dot\gamma(a)}|_{\Ps\times \Ps}$ and $g_{\dot\gamma(b)}|_{\tilde \Ps\times \tilde \Ps}$ are nondegenerate,
%Assume that the Legendre transformation $\mathscr{L}_L$ is injective and 
 consider a smooth $A$-admissible $(\Ps,\tilde \Ps)$-variation (all the curves in the variation start in $\Ps$ and end in $\tilde \Ps$). Then 
\begin{multline*}
E''(0)=\int_a^b \left(-g_{\dot\gamma}(\hat R_{\dot\gamma}(\dot\gamma,W)W,\dot\gamma)+g_{\dot\gamma}(\hat D_\gamma^{\dot\gamma}W,\hat D_\gamma^{\dot\gamma}W)\right)dt\\+
g_{\dot\gamma(b)}(\hat S^\Ps_{\dot\gamma(b)}(W,W),\dot\gamma(b))
-g_{\dot\gamma(a)}(\hat S^{\tilde \Ps}_{\dot\gamma(a)}(W,W),\dot\gamma(a)),
\end{multline*}
where $W$ is the variational vector field of the variation along $\gamma$. % then
\end{prop}
\begin{proof}
The formulas can be obtained for example as in \cite[Prop. 3.1 and 3.2 and Cor. 3.8]{JaSo14}  with one exception, since in  \cite[Prop. 3.2]{JaSo14}, $g_{\dot\gamma}(\hat R_{\dot\gamma}(\dot\gamma,W)W,\dot\gamma)$ is replaced with $g_{\dot\gamma}( \hat R^\gamma(\dot\gamma,W)W,\dot\gamma)$. Observe that from \eqref{gP}, % and the property \eqref{vertprop} for $\hat P$, i
t follows that  
\[g_{\dot\gamma}(\hat R_{\dot\gamma}(\dot\gamma,W)W,\dot\gamma)=g_{\dot\gamma}( \hat R^\gamma(\dot\gamma,W)W,\dot\gamma),\]
(recall that $\hat R^\gamma$ coincides with $\hat R_\Lambda$ defined just before \eqref{RLambda} without the $P$-terms), which concludes.
%Observe that 
%\begin{equation}\label{formulaP}
%P_v(u,w,z)=B_v(u,w,z)+(\partial^\nu\mathfrak{L}^\flat)_v(u,w,z)
%\end{equation}
%and then $g_v(P_v(v,w,z),v)=0$ (apply the last identity together with \eqref{partialL}, \eqref{landsberg} and \eqref{Land0}). This can be used in \cite[Eq. (13)]{JaSo14} in order to replace $R^\gamma(\dot\gamma,W)W$ with $R_{\dot\gamma}(\dot\gamma,W)W$. Moreover, for the proof of \cite[Prop. 3.11]{JaSo14}, we can use \eqref{symR}, which gives $g_{\dot\gamma}(R_{\dot\gamma}(\dot\gamma,V)W,\dot\gamma)=-g_{\dot\gamma}(R_{\dot\gamma}(\dot\gamma,V)\dot\gamma,W)$, allowing one to conclude that $V$ satisfies $V''=R_{\dot\gamma}(\dot\gamma,V)\dot\gamma$. Finally, \eqref{formulaP} together with \eqref{partialL}, \eqref{landsberg} and \eqref{Land0} allows one to obtain straightforwardly \cite[Lemma 1.2]{Ja14b}.
\end{proof}
 
\subsection{The osculating metric}\label{sub-osculating}
If we fix a vector field $V$ in an open subset $\Omega\subset M$, then we can consider the osculating metric $g_V$ and its Levi-Civita connection $\overline{\nabla}$. Let us compare now  both connections. In the particular case of the Chern connection, this can be found for example in \cite[Prop. 8.4.3]{Sh01}.
\begin{prop}\label{oscul}
Given an $A$-admissible vector field $V\in \mathfrak{X}(\Omega)$, with $\Omega$ an open subset of a manifold $M$, and a pseudo-Finsler metric $L:A\rightarrow \R$, let $\overline{\nabla}$ be the Levi-Civita connection of $g_V$ and $\hat\nabla$ satisfying \eqref{nablag=Q}. Then   
\begin{multline*}
g_V(\hat\nabla^V_XY-\overline{\nabla}_XY,Z)=-C_V(Y,Z,\hat\nabla^V_XV)-C_V(Z,X,\hat\nabla^V_YV)\\+C_V(X,Y,\hat\nabla^V_ZV)
-\frac 12\Q_V(X,Y,Z).
\end{multline*}
In particular,
\begin{align*}
g_V(\hat\nabla^V_XV-\overline{\nabla}_XV,Z)&=-C_V(Z,X,\hat\nabla^V_VV),\\
g_V(\hat\nabla^V_VX-\overline{\nabla}_VX,Z)&=-C_V(X,Z,\hat\nabla^V_VV),\\
g_V(\hat\nabla^V_XY-\overline{\nabla}_XY,V)&=C_V(X,Y,\hat\nabla^V_VV).
\end{align*}
When $V$ is a geodesic vector field, then $\hat\nabla^V_XV=\overline{\nabla}_XV$, $\hat\nabla^V_VX=\overline{\nabla}_VX$ and 
\[{\hat R}_V(V,X)V= \overline{R}(V,X)V,\]
where $\hat R$ and $\overline{R}$ are the curvature tensors associated with $\hat \nabla$ and $\overline{\nabla}$, respectively. 
\end{prop}
\begin{proof}
The formulas for the difference between $\hat\nabla$ and $\overline{\nabla}$ are a consequence of the Koszul formula \eqref{KoszulQ}. For the equality between the curvature tensors, observe that as the vertical derivative of $\hat\nabla$ satisfies \eqref{vertprop}, then using that $V$ is a geodesic vector field and the relations between $\hat\nabla$ and $\overline{\nabla}$, it follows that
\[\hat R_V(V,X)V=\hat\nabla^V_V\hat\nabla^V_XV-\hat\nabla^V_{[X,V]}V=\bar R(V,X)V.\]
\end{proof}
\section*{Acknowledgments}
The author warmly acknowledges useful discussions with Professors Amir Aazami (Clark University, USA) and Eduardo Mart\'inez (University of Zaragoza, Spain), as well as some improvements suggested by the referees.
 
 The research is a result of the activity developed within the framework of the Programme in Support of Excellence Groups of the Regi\'on de Murcia, Spain, by Fundaci\'on S\'eneca, Science and Technology Agency of the Regi\'on de Murcia.  It was partially supported by MINECO/FEDER project MTM2015-65430-P, MICINN/FEDER project  PGC2018-097046-B-I00  and Fundaci\'on S\'eneca project 19901/GERM/15, Spain.

 %\bibliographystyle{siam}
 %\bibliography{mybib}

\begin{thebibliography}{10}
 	
 	\bibitem{BCS00}
 	{\sc D.~Bao, S.-S. Chern, and Z.~Shen}, {\em An introduction to
 		{R}iemann-{F}insler geometry}, vol.~200 of Graduate Texts in Mathematics,
 	Springer-Verlag, New York, 2000.
 	
 	
 	
 	\bibitem{Jav14a}
 {\sc M.~A. Javaloyes}, { Chern connection of
 		a pseudo-{F}insler metric as a family of affine connections}, {\em Publ. Math.
 	Debrecen,} 84 (2014), pp.~29--43.
 	
 	\bibitem{Ja14b}
 	{\sc M.~A. Javaloyes}, { Corrigendum to
 		``{C}hern connection of a pseudo-{F}insler metric as a family of affine
 		connections'' [mr3194771]}, {\em Publ. Math. Debrecen}, 85 (2014), pp.~481--487.
 	\bibitem{Ja19}
 	
 	{\sc M.~A. Javaloyes}, { Anisotropic tensor calculus}, {\em  Int. J. Geom. Methods
 	Mod. Phys.} 16  (2019), suppl. 2, 1941001, 26 pp.
 	
 	\bibitem{JaSo14}
 	{\sc M.~A. Javaloyes and B.~L. Soares}, { Geodesics and {J}acobi fields of
 		pseudo-{F}insler manifolds}, {\em Publ. Math. Debrecen}, 87 (2015), pp.~57--78.
 	
 	\bibitem{KN63}
 	{\sc S.~Kobayashi and K.~Nomizu}, {\em Foundations of differential geometry.
 		{V}ol. {I}}, Wiley Classics Library, John Wiley \& Sons, Inc., New York,
 	1996.
 	\newblock Reprint of the 1963 original, A Wiley-Interscience Publication.
 	
 	\bibitem{Edu1}
 	{\sc E.~Mart\'{\i}nez, J.~F. Cari\~{n}ena, and W.~Sarlet}, { Derivations of
 		differential forms along the tangent bundle projection}, {\em Differential Geom.
 	Appl.}, 2 (1992), pp.~17--43.
 	
 	\bibitem{Edu2}
 	{\sc E.~Mart\'{\i}nez, J.~F. Cari\~{n}ena, and W.~Sarlet}, { Derivations of
 		differential forms along the tangent bundle projection. {II}}, {\em Differential
 	Geom. Appl.}, 3 (1993), pp.~1--29.
 	
 	\bibitem{Mat80}
 	{\sc H.-H. Matthias}, {\em Zwei {V}erallgemeinerungen eines {S}atzes von
 		{G}romoll und {M}eyer}, Bonner Mathematische Schriften [Bonn Mathematical
 	Publications], 126, Universit\"at Bonn, Mathematisches Institut, Bonn, 1980.
 	\newblock Dissertation, Rheinische Friedrich-Wilhelms-Universit{\"a}t, Bonn,
 	1980.
 	
 	\bibitem{Oneill}
 	{\sc B.~O'Neill}, {\em Semi-{R}iemannian geometry}, vol.~103 of Pure and
 	Applied Mathematics, Academic Press, Inc. [Harcourt Brace Jovanovich,
 	Publishers], New York, 1983.
 	\newblock With applications to relativity.
 	
 	\bibitem{Rund}
 	{\sc H.~Rund}, {Direction-dependent connection and curvature forms}, {\em Abh.
 	Math. Sem. Univ. Hamburg,} 50 (1980), pp.~188--209.
 	
 	\bibitem{Sh01}
 	{\sc Z.~Shen}, {\em Differential geometry of spray and {F}insler spaces},
 	Kluwer Academic Publishers, Dordrecht, 2001.
 	
 \end{thebibliography}
 
 \end{document}